\documentclass[12pt,leqno]{article}
\usepackage{amsfonts}
\usepackage{amsmath, amsthm, amsfonts, amssymb, color}
\usepackage{mathrsfs}
\renewcommand{\bar}{\overline}
\renewcommand{\hat}{\widehat}
\renewcommand{\tilde}{\widetilde}

\newtheorem{thm}{Theorem}[section]

\newtheorem{lem}[thm]{Lemma}

\theoremstyle{definition}
\newtheorem{defn}{Definition}[section]
\newcommand{\scr}[1]{\mathscr #1}
\definecolor{wco}{rgb}{0.5,0.2,0.3}

\numberwithin{equation}{section} \theoremstyle{remark}
\newtheorem{rem}{Remark}[section]

\newcommand{\ua}{\uparrow}

\title{{\bf Asymptotic Log-Harnack Inequality and Applications for Stochastic Systems of Infinite Memory}\thanks{Supported in
 part by  NNSFC (11771326, 11431014,11831014) and a Co-Fund grant.}
}
\author{
{\bf  Jianhai Bao$^{b),c)}$,  Feng-Yu Wang$^{a),c)}$, Chenggui Yuan$^{c)}$}\\
\footnotesize{$^{a)}$Center of Applied Mathematics, Tianjin
University, Tianjin 300072, China}\\
\footnotesize{$^{b)}$School of Mathematics and Statistics, Central
South
University, Changsha 410083, China}\\
\footnotesize{$^{c)}$Department of Mathematics, Swansea University,
Singleton Park, SA2 8PP, UK}\\ \footnotesize{jianhai.bao@swansea.ac.uk,
wangfy@tju.edu.cn, C.Yuan@swansea.ac.uk}}

\date{}
\begin{document}
\def\R{\mathbb R}  \def\ff{\frac} \def\ss{\sqrt} \def\B{\mathbf
B}
\def\N{\mathbb N} \def\kk{\kappa} \def\m{{\bf m}}
\def\dd{\delta} \def\DD{\Delta} \def\vv{\varepsilon} \def\rr{\rho}
\def\<{\langle} \def\>{\rangle} \def\GG{\Gamma} \def\gg{\gamma}
  \def\nn{\nabla} \def\pp{\partial} \def\EE{\scr E}
\def\d{\text{\rm{d}}} \def\bb{\beta} \def\aa{\alpha} \def\D{\scr D}
  \def\si{\sigma} \def\ess{\text{\rm{ess}}}
\def\beg{\begin} \def\beq{\begin{equation}}  \def\F{\scr F}
\def\Ric{\text{\rm{Ric}}} \def\Hess{\text{\rm{Hess}}}
\def\e{\text{\rm{e}}} \def\ua{\underline a} \def\OO{\Omega}  \def\oo{\omega}
 \def\tt{\tilde} \def\Ric{\text{\rm{Ric}}}
\def\cut{\text{\rm{cut}}} \def\P{\mathbb P} \def\ifn{I_n(f^{\bigotimes n})}
\def\C{\scr C}      \def\aaa{\mathbf{r}}     \def\r{r}
\def\gap{\text{\rm{gap}}} \def\prr{\pi_{{\bf m},\varrho}}  \def\r{\mathbf r}
\def\Z{\mathbb Z} \def\vrr{\varrho} \def\ll{\lambda}
\def\L{\scr L}\def\Tt{\tt} \def\TT{\tt}\def\II{\mathbb I}
\def\i{{\rm in}}\def\Sect{{\rm Sect}}\def\E{\mathbb E} \def\H{\mathbb H}
\def\M{\scr M}\def\Q{\mathbb Q} \def\texto{\text{o}} \def\LL{\Lambda}
\def\Rank{{\rm Rank}} \def\B{\scr B} \def\i{{\rm i}} \def\HR{\hat{\R}^d}
\def\to{\rightarrow}\def\l{\ell}
\def\8{\infty}\def\X{\mathbb{X}}\def\3{\triangle}
\def\V{\mathbb{V}}\def\M{\mathbb{M}}\def\W{\mathbb{W}}\def\Y{\mathbb{Y}}\def\1{\lesssim}
\def\Lip{{\rm Lip}}

\def\La{\Lambda}\def\S{\mathbf{S}}

\renewcommand{\bar}{\overline}
\renewcommand{\hat}{\widehat}
\renewcommand{\tilde}{\widetilde}
 \maketitle

\begin{abstract}
The asymptotic log-Harnack inequality is established  for several
kinds of models on stochastic differential systems with infinite
memory: non-degenerate SDEs, neutral SDEs, semi-linear SPDEs, and
stochastic Hamiltonian systems.
  As applications,      the following properties are derived for the associated segment Markov semigroups:    asymptotic heat
kernel estimate, uniqueness of the invariant probability measure,
asymptotic gradient estimate (hence,
  asymptotically strong Feller property), as well as  asymptotic irreducibility. 

\end{abstract}
 \vskip 2cm

\section{Introduction}

The dimension-free Harnack inequality  was initiated  in
\cite{W97} for elliptic diffusion semigroups on Riemannian
manifolds. In case such kind of inequality is unavailable, the
log-Harnack inequality was introduced alternatively in \cite{W10}.
Both inequalities have been investigated extensively and applied to
(singular, degenerate) SDEs/SPDEs via coupling by change of measures
developed in e.g. \cite{ATW06, W07}; see \cite{Wbook} and references
within for more details. In particular, these inequalities imply
gradient estimates (hence, the strong Feller property), the
uniqueness of invariant probability measures, heat kernel estimates,
and irreducibility  of the associated Markov semigroups.

However, when the stochastic system is  highly degenerate so that
these properties are unavailable, the above type Harnack
inequalities no longer hold. In this scenario, it is natural to
investigate weaker versions of these properties by exploiting
 Harnack inequalities in the weak version. For instance, the strong
Feller property is invalid for degenerate stochastic $2D$
Navier-Stokes equations, whereas the weaker $``$asymptotically
strong Feller"  property has been proved in \cite{HM} and \cite{Xu}
by making use of asymptotic couplings and ``modified log-Harnack
inequality", respectively. Since the log-Harnack inequality in the
weak version is concerned with long time behavior, below we shall
call it
  ``asymptotic log-Harnack inequality".

In this paper, we aim to investigate asymptotic log-Harnack
inequality and its applications for SDEs with infinite memory, i.e.,
the coefficients of   SDEs involved depend on the whole history of
the system. In this setup, the strong Feller property is invalid (see e.g. \cite{BS, HMS}), so we are in the weak
situation  without log-Harnack inequalities. When the memory is
finite and the noise is path-independent, the dimension-free Harnack
inequality, log-Harnack inequality and gradient estimates have been
investigated in \cite{BWY,BWYb, EVS, SWY, WY}, to name a few.

Before considering  specific models, in Section 2 we  present some
applications of the asymptotic log-Harnack inequality in a general
framework, which are new except the asymptotically strong Feller
property   derived in \cite{Xu}. In Sections 3-6, we establish
 asymptotic log-Harnack inequality for the following
  stochastic differential systems with infinite memory, respectively,
  including
non-degenerate SDEs, neutral SDEs, semi-linear SPDEs, and stochastic
Hamiltonian systems. In the Appendix section, we address the
existence and uniqueness of solutions to SDEs with infinite memory
under the locally weak monotone condition and the weak coercive
condition.


\section{Applications of  asymptotic log-Harnack inequality}
Before we recall the definition on asymptotically strong Feller
introduced in \cite{HM} for  a Markov semigroup
 $P_t$, we start with some   notation and notions.
Let $(E, \rr)$ be a metric space,   $\B_b(E)$  the class of bounded
measurable functions on $E$, and   $\B^+_b(E)$   the set of positive
functions in
 $\B_b(E).$
A continuous function $d: E\times E\to\R_+:=[0,\8)$ is called a
pseudo-metric if $d(x,x)=0$ and $d(x,y)\le d(x,z)+d(z,y)$ hold  for
$x,y,z\in E$. For a  pseudo-metric $d$,   the transportation cost
(which also is called $L^1$-Wasserstein distance when $d$ is a
distance) is defined by
$$W_1^d(\mu_1,\mu_2)= \inf_{\pi\in \scr C(\mu_1,\mu_2)} \int_{E\times E} d(x,y)\pi(\d x,\d y),\ \ \mu_1,\mu_2\in \scr P(E),$$
where $\scr P(E)$ stands for the class of probability measures on
$E$, and $\scr C(\mu_1,\mu_2)$ consists of all couplings of $\mu_1$
and $\mu_2$, that is, $\pi\in \scr C(\mu_1,\mu_2)$ means  $\pi\in
\scr P(E\times E)$ with $\pi(\cdot\times E)=\mu_1$ and
$\pi(E\times\cdot)=\mu_2.$ An increasing  sequence of pseudo-metrics
$(d_n)_{n\ge1}$ (i.e., $d_i(\cdot,\cdot)\ge d_j(\cdot,\cdot), i\ge
j$) is said to be   a totally separating  system if
$\lim_{n\to\8}d_n(x,y)=1$ for all $x\neq y$.

\begin{defn}\label{ASF} The Markov semigroup $P_t$ is called asymptotically strong Feller at a point
$x\in E,$ if there exist  a totally separating system of
pseudo-metrics $(d_k)_{k\ge 1}$  and a sequence $t_k\uparrow\infty$
such that
\begin{equation}\label{eq5}
\inf_{U\in\mathcal {U}_x}\limsup_{k\rightarrow\8}\sup_{y\in
U} W_1^{d_k}( P_{t_k}(x,\cdot), {P}_{t_k}(y,\cdot))=0,
\end{equation}
where $\mathcal {U}_x$ denotes the collection of all open sets
containing $x$, and $P_t(x, A):= P_t1_A(x)$ for $x\in E$ and a
measurable set $A\subset E$. $P_t$  is called asymptotically strong
Feller if it is asymptotically strong Feller at any $x\in E$.
\end{defn}

For a function $f:E\to\R$, define
$$|\nn f|(x)=\limsup_{y\to x} \ff{|f(x)-f(y)|}{\rr(x,y)},\ \ x\in E.$$
Let $\|\cdot\|_\infty$ be the uniform norm. So, $\|\nn f\|_\infty =
\sup_{x\in E} |\nn f|(x).$
 Set  $\Lip(E):=\{f: E\to \R, \|\nn f\|_\infty<\infty\}$,   the family of all Lipschitzian functions on
 $E$. Next, we introduce the asymptotic log-Harnack inequality.

 \beg{defn} The following inequality is called an asymptotic log-Harnack inequality of $P_t$:
\beq\label{ALH}  P_t \log f(x)\le \log P_t f(y) +   \Phi(x,y) +
\Psi_t(x,y)\|\nn\log f\|_\infty,\ \ t>0
\end{equation} for any $f\in \B_b^+(E)$ with $\|\nn\log f\|_\infty<\8$, where $\Phi, \Psi_t: E\times E\to (0,\infty)$ are measurable with  $\Psi_t\downarrow 0$ as $t\uparrow \infty$. \end{defn}

Below, we present some asymptotic properties implied by \eqref{ALH}.
For a  measurable set $A\subset E$ and $x\in E$, let
$\rr(x,A)=\inf_{y\in A} \rr(x,y)$, i.e.,  the distance between $x$
and $A$. Moreover, for any $\vv>0$, let  $A_\vv=\{y\in
E:\rr(y,A)<\vv\}$ and $A_\vv^c$ be the complement of $A_\vv.$

\beg{thm}\label{T2.1} Let $P_t$ satisfy $\eqref{ALH}$ for some
  symmetric functions $\Phi, \Psi_t: E\times E\to \R_+$
with   $\Psi_t\downarrow 0$ as $t\uparrow \infty$. Then:
\beg{enumerate}
\item[$({\bf1})$]{\bf (Gradient estimate)} If,  for any $x\in E$,
\begin{equation}\label{c7}
\Lambda(x):=\limsup_{y\to
x}\ff{\Phi(x,y)}{\rr(x,y)^2}<\infty,  ~~\mbox{ and
}~~\Gamma_t(x)  :=\limsup_{y\to
x}\ff{\Psi_t(x,y)}{\rr(x,y)}<\infty,
\end{equation} then, for any $t>0$ and $ f\in
\Lip_b(E):=\Lip(E)\cap\B_b(E)$, \beq\label{GES}|\nn  {P}_tf|
\le\ss{2\LL}\ss{ {P}_tf^2 -( {P}_tf)^2}+\|\nn
f\|_\8\Gamma_t.\end{equation} In particular, when
$\Gamma_t\downarrow 0$ as $t\uparrow \infty$, $P_t$ is asymptotically
strong Feller.

 \item[$({\bf2})$]{\bf (Asymptotic heat kernel estimate)}  If $P_t$ has an invariant probability measure $\mu$, then, for any $f\in\B^+_b(E)$ with $\|\nn f\|_\8<\8$,
\begin{equation}\label{03}
\limsup_{t\to\8}P_tf(x)\le
\log\bigg(\ff{\mu(\e^f)}{\int_E\e^{-\Phi(x,y)}\mu(\d
y)}\bigg),~~~~x\in E.
\end{equation} Consequently, for any closed set $A\subset E$ with $\mu(A)=0,$
\begin{equation}\label{05}
\lim_{t\to\8}P_t1_A(x)=0,~~~~x\in E.
\end{equation}
\item[$({\bf3})$]{\bf (Uniqueness of invariant probability)} $P _t$ has at most one   invariant probability measure.
\item[$({\bf4})$]{\bf (Asymptotic irreducibility)} Let $x\in E$ and  $A\subset E$ be a measurable set such that
$$\dd(x,A):= \liminf\limits_{t\to\infty} P_t(x,A)>0.$$   Then,
\beq\label{AI} \liminf_{t\to\infty} P_t(y,A_\vv)>0,\ \ y\in E,\,
\vv>0.\end{equation} Moreover, for any $\vv_0\in (0,\dd(x,A)),$
there exists a constant $t_0>0$ such that \beq\label{AI2}
P_t(y,A_\vv)>0\ \text{ provided }\  t\ge t_0,~~
\Psi_t(x,y)<\vv\vv_0.\end{equation}
 \end{enumerate}\end{thm}

According to the proof of   \cite[Theorem 1.4.1(4)]{Wbook}, if
\eqref{03} holds without  limit but for a fixed $t>0$, then $P_t$
has a density $p_t(x,y)$ with respect to $\mu$ satisfying the
entropy estimate
$$\int_E p_t(x,y)\log p_t(x,y)\mu(\d y) \le -\log \int_E \e^{-\Phi(x,y)}\mu(\d y).$$
So,   \eqref{03} can be regarded as    the asymptotic heat kernel
estimate of $P_t$.

\begin{proof}[Proof of Theorem \ref{T2.1}] ({\bf1}) In terms of \cite{Xu}, if $\Gamma_t\downarrow 0$ as $t\uparrow \infty$, \eqref{GES} implies the asymptotically strong Feller property. So, it suffices to prove the gradient estimate \eqref{GES}.

For any $x\in E, t>0$ and $f\in \Lip_b(E)$,   we take $x_n\to x$
such that $\vv_n:=\rr(x_n,x)\downarrow 0$ and  (in case the limit
below is negative, write $-f$ instead of $f$) \beq\label{GG0} |\nn
P_tf|(x) = \limsup_{n\to\infty} \ff{P_t f(x_n)-P_t
f(x)}{\vv_n}.\end{equation} For any constant $c>0$,
  \eqref{ALH} implies
\begin{equation}\label{c4}
P_t\log (1+c\,\vv_n f)(x_n)\le \log {P}_t(1+c\,\vv_n
f)(x)+\Phi(x_n,x)+\Psi_t(x_n,x)\|\nn \log (1+c\vv_n f)\|_\8.
\end{equation}
By Taylor's expansion, for $\vv_n$ sufficiently small  we have
 \begin{equation}\label{c2}
P_t\log (1+c\,\vv_n f)(x_n)=c\,\vv_n {P}_tf(x_n)-\ff{c^2\vv^2_n}{2}P_tf^2(x_n)+o(\vv^2_n),
\end{equation}
and
\begin{equation}\label{c3}
 \log P_t(1+c\,\vv_n f)(x)=c\,\vv_n {P}_tf(x)-\ff{c^2\vv^2_n}{2}(P_tf)^2(x)+o(\vv^2_n).
\end{equation}
Substituting \eqref{c2} and \eqref{c3} into \eqref{c4} yields
\begin{equation*}
\begin{split}
&c\,\vv_n (P_tf(x_n)-P_tf(x))\le
\ff{c^2\vv^2_n}{2}\Big\{{P}_tf^2(x_n)-(P_tf)^2(x)\Big\}\\
&\quad+\Phi(x_n,x)+\Psi_t(x_n,x)\|\nn \log (1+c\,\vv_n
f)\|_\8+o(\vv^2_n)\\
&\le\ff{c^2\vv^2_n}{2}\Big\{P_tf^2(x_n)-({P}_tf)^2(x)\Big\}+\Phi(x_n,x)+c\,\vv_n\Psi_t(x_n,x)\|\nn
f\|_\8+o(\vv^2_n).
\end{split}
\end{equation*} Combining this with \eqref{GG0}, we obtain
\begin{equation*}
\begin{split}
|\nn P_t f|(x)
&\le\limsup_{n\to\infty}\bigg(\ff{c}{2}\Big\{{P}_tf^2(x_n)-({P}_tf)^2(x)\Big\}+\ff{\Phi(x_n,x)}{c\,\vv_n^2}+\ff{\Psi_t(x_n,x)}{\vv_n}\|\nn
f\|_\8\bigg)\\
&\le \ff c 2 \Big\{\limsup_{n\to\infty} P_t f^2(x_n)
-(P_tf)^2(x)\Big\} + \ff {\LL(x)}c + \|\nn
f\|_\infty\GG_t(x).\end{split}\end{equation*} This, in particular,
implies $P_t \Lip_b(E)\subset C_b(E)$, so that $\limsup_{n\to\infty}
P_t f^2(x_n) = P_tf^2(x)$. Consequently,
\begin{equation*}
 |\nn P_t f|(x)\le\ff{c}{2}\Big\{P_tf^2(x)-({P}_tf)^2(x)\Big\}+\ff{\Lambda(x)}{c}+\|\nn
f\|_\8\Gamma_t(x),\ \ c>0.
 \end{equation*} Minimizing the upper bound with respect to $c>0$, we therefore obtain \eqref{GES}.

({\bf2}) Applying \eqref{ALH} to $f=\e^g$ for $g\in\B_b(E)$ with
$\|\nn g\|_\8<\8$, we infer that
\begin{equation*}
\begin{split}
P_t  g(x)\le \log({P}_t\e^{g(y)})+\Phi(x,y)+\Psi_t(x,y)\|\nn
g\|_\8,~~~~x,y\in E.
\end{split}
\end{equation*}
Equivalently,
\begin{equation*}
\exp\Big(P_t  g(x)-\Phi(x,y)-\Psi_t(x,y)\|\nn g\|_\8\Big)\le
P_t\e^{g(y)},~~~~x,y\in E.
\end{equation*}
Integrating  with respect to $\mu(\d y)$ on both sides and
exploiting the $P_t$-invariance of $\mu$, we thus derive
\begin{equation*}
\e^{P_t g(x)}\int_E  \exp\Big(-\Phi(x,y)-\Psi_t(x,y)\|\nn g\|_\8\Big)\mu(\d y)\le\mu(\e^g),~~~~~x\in E.
\end{equation*}
Hence,
\begin{equation*}
P_t g(x)\le\log\bigg(\ff{\mu(\e^g)}{\int_E
\exp\Big(-\Phi(x,y)-\Psi_t(x,y)\|\nn g\|_\8\Big)\mu(\d
y)}\bigg),~~~~x\in E.
\end{equation*}
Whence,   \eqref{03} follows  by taking $t\rightarrow\8.$

Next, for a closed  set $A\subset E$  with $\mu(A)=0,$ let
\begin{equation*}
g_k=(1-k\rr(\cdot,A))^+,~~~k\ge1.
\end{equation*}
Then  $g_k\downarrow 1_A$ as $k\uparrow\8$. Since $g_k|_A=1$ and $
g_k\ge 0,$ we have
\begin{equation*}
m\limsup_{t\to\8}P_t1_A(x)\le
\limsup_{t\to\8}{P}_t(mg_k(x)),~~~~~~x\in E, m\ge 1.
\end{equation*}
This, together with \eqref{03},  leads to
\begin{equation}\label{04}
 \limsup_{t\to\8} {P}_t1_A(x)\le\ff{1}{m}\log\bigg(\ff{\mu(\e^{mg_k})}{\int_E  \e^{-\Phi(x,y)}\mu(\d y)}\bigg),~~~~~~x\in E, m\ge 1.
\end{equation}
Due to $\mu(A)=0$, one has
\begin{equation*}
\begin{split}
\lim_{k\to \infty} \mu(\e^{mg_k})=\mu(\e^{m1_A})&=\int_A\e^{m1_A(x)}\mu(\d
x)+\int_{A^c}\e^{m1_A(x)}\mu(\d x)\\
&=\e^m\mu(A)+\mu(A^c)=1
\end{split}
\end{equation*} so that, by taking $k\to\8$  in \eqref{04}, we arrive
at
\begin{equation*}
 \limsup_{t\to\8}{P}_t1_A(x)\le\ff{1}{m}\log\bigg(\ff{1}{\int_E  \e^{-\Phi(x,y)}\mu(\d y)}\bigg),~~~~~~x\in E.
\end{equation*}
Therefore, \eqref{05} holds true by approaching $m\to\8.$

({\bf3}) Since the class of invariant probability measures of $P_t$
is convex, and any two different extreme measures in the class are
mutually singular (see e.g. \cite[Proposition 3.2.5]{DZ}), it
suffices to show that any two invariant probability measures
$\mu,\tt\mu$ are equivalent. For any measurable  set $A\subseteq E$
with $\mu(A)=0$, we aim to prove $\tt\mu(A)=0$. Let $\tt A\subset A$
be a closed set. By the $P_t$-invariance of $\tt\mu$, \eqref{05} and
Fatou's lemma, we obtain
\begin{equation*}
\tt\mu(\tt A)=\limsup_{t\to\8}\tt\mu(P_t1_{\tt A})\le
\tt\mu\Big(\limsup_{t\to\8}P_t1_A\Big)=0.
\end{equation*}
So, one has
\begin{equation*}
\tt\mu(A)=\sup_{\tt A\subset A,\tt A \mbox{ closed}}\tt\mu(\tt A)=0.
\end{equation*}
As a consequence, we conclude that $\tt \mu$ is absolutely
continuous with respect to $\mu$. Similarly,  we can infer that
$\mu$ is absolutely continuous with respect to $\tt \mu$.

({\bf4}) Let $f(z)= \vv^{-1} (\vv-\rr(z, A))^+.$  Then we have
$f|_A=1, f|_{A_\vv^c}=0$ and $\|\nn f\|_\infty= \vv^{-1}.$ So, for
any $n\ge 1$,    \eqref{ALH} implies
\begin{equation*}\beg{split}& nP_t(x,A)\le P_t \log (\e^{nf})(x)\le \log P_t(\e^{nf})(y)+\Phi(x,y)+n \Psi_t(x,y)\|\nn f\|_\infty\\
&\le \log\big(1+\e^nP_t(y, A_\vv)\big) +\Phi(x,y)+n\,\vv^{-1}
\Psi_t(x,y).\end{split}\end{equation*} So we have
\begin{equation}\label{AI20}
 P_t(x,A)\le  \ff{1}{n} \log\big(1+\e^n P_t(y, A_\vv)\big)
+\ff{1}{n}\Phi(x,y)+\vv^{-1} \Psi_t(x,y). \end{equation} If
$\liminf_{t\to\infty} P_t(y,A_\vv)=0$, then, in \eqref{AI20}, taking
$t\to\8$ followed by letting $n\to\8$ and
 using $\Psi_t\to 0$ as
$t\to\infty$ yields
$$\dd(x,A)= \liminf_{t\to\infty}P_t(x,A) \le 0,$$
which
 contradicts to $\dd(x,A)>0$. Henceforth, \eqref{AI} holds.

Next, take $t_0>0$ such that $P_t(x,A)\ge \vv_0$ holds for all $t\ge
t_0$. So, if $t\ge t_0$ such that $P_t(y, A_\vv)=0$, then, due to
\eqref{AI20} by taking $n\to\8$, we have
$$\vv_0\le P_t(x,A)\le   \vv^{-1} \Psi_t(x,y).$$ Hence, \eqref{AI2} holds.
\end{proof}

\section{Non-degenerate SDEs of infinite memory}

Let $(\R^d,\<\cdot,\cdot\>,|\cdot|)$ be the $d$-dimensional
Euclidean space. $\C=C((-\8,0];\R^d)$ denotes the family of all
continuous functions $f:(-\8,0]\rightarrow\R^d$. For a fixed
constant $r>0$, set
\begin{equation}\label{a1}
\mathscr{C}_r:=\Big\{\xi\in
\C:\|\xi\|_r:=\sup_{-\8<\theta\le0}(\e^{r\theta}|\xi(\theta)|)<\8\Big\},
\end{equation}
which is a Polish (i.e., complete, separable, metrizable) space with
the norm $\|\cdot\|_r$. Since $r>0$ and $\theta\le 0$, the norm
$\|\cdot\|_r$ means that the influence of history is exponentially
weak with respect to the time parameter, which is a natural feature
in the real world.

Let $\mathcal
{M}_0=\mathcal {M}_0((-\8,0])$ be the set of all
probability measures on $(-\8,0].$ For $\kk>0$, set
\begin{equation*}
\mathcal {M}_\kk:=\Big\{\mu\in\mathcal
{M}_0:\mu^{(\kk)}:=\int_{-\8}^0\e^{-\kk\theta}\mu(\d\theta)<\8\Big\}.
\end{equation*}
$\R^d\otimes\R^d$ stands for the set of all $n\times n$-matrices
with real entries, which is equipped with the Hilbert-Schmidt norm
$\|\cdot\|_{ HS}$. Denote $\B_b(\C_r)$ by the family of all bounded
measurable functions $\phi:\C_r\rightarrow\R$ with the uniform norm
$\|\phi\|_\8:=\sup_{\xi\in\C_r}|\phi(\xi)|$. For  $A\in
\R^d\otimes\R^d$, let $A^{-1}$ be its inverse (if it exists) and
 $\|A\|$   its operator norm.

We  consider the following    SDE with infinite memory:
\begin{equation}\label{eq0}
\d X(t)=b(X_t)\d t+\si(X_t)\d W(t),~~~~t>0,~~~~X_0=\xi\in\C_r,
\end{equation}
where, for each fixed $t\ge0$, $X_t(\cdot)\in\C_r$  is defined
by $$X_t(\theta):=X(t+\theta),\ \ \theta\in( -\8,0],$$ which is called the segment process of $X(t)$,
$b:\C_r\rightarrow\R^d,$ $\si:\C_r\to\R^d\otimes\R^d$, and
$(W(t))_{t\ge0}$ is a  $d$-dimensional Brownian motion   on
some complete filtered probability space $(\OO,\F,(\F_t)_{t\ge0},\P)$.

To ensure   existence and uniqueness of solutions to \eqref{eq0} and
to establish the asymptotic log-Harnack inequality, we impose the
following assumptions on the coefficients $b$ and $\si$:
\begin{enumerate}
\item[({\bf H1})] $b\in C(\C_r)$ is  bounded on bounded subsets of $\C_r$, and there exists $K_1>0$ such that
\begin{equation}\label{b2}
2\<\xi(0)-\eta(0),b(\xi)-b(\eta)\>\le
K_1\|\xi-\eta\|_r^2,~~~~\xi,\eta\in\C_r;
\end{equation}

\item[({\bf H2})]  There exists $K_2>0$ such that
\begin{equation*}
\|\si(\xi)-\si(\eta)\|_{\rm HS}^2\le
K_2\|\xi-\eta\|_r^2,~~~~\xi,\eta\in\C_r;
\end{equation*}

\item[({\bf H3})]  $\|\si\|_\infty:=\sup\limits_{\xi\in\C_r}\|\si(\xi)\|<\8$, and   $\si$ is invertible with
 $\|\si^{-1}\|_\infty:=\sup\limits_{\xi\in\C_r}\|\si^{-1}(\xi)\|<\8$.
\end{enumerate}

These two assumptions guarantee that \eqref{eq0} admits a
 unique   solution $(X^\xi(t))_{t\ge 0}$ with the initial value
$X_0=\xi\in \C_r$; For see Theorem \ref{existence} below in detail.
Moreover, the segment process (or functional solution)
$(X_t^\xi)_{t\ge0}$ enjoys
 the Markov property; see e.g. \cite[Theorem 4.2]{WYM}.
So,
$$P_t f(\xi)= \E f(X_t^\xi),\ \ t\ge 0, \, \xi\in \C_r,\,  f\in \B_b(\C_r)$$ gives rise to a Markov semigroup $P_t$. Since the memory is infinite,
 $P_t$ is not strong Feller; see, for instance, \cite{BS,HMS}.

Assumption ({\bf H3}) is the usual ellipticity condition and will be used  to construct couplings by change of measures for  asymptotic log-Harnack inequalities.

\begin{thm}\label{th1} Assume
{\rm  ({\bf H1})-({\bf H3})}. For any  $r_0\in(0,r)$, there exists a constant $c>0$ such that
\begin{equation}\label{eq6}
P_t\log f(\eta)  \le\log P_t f(\xi)+c\,\|\xi-\eta\|_r^2+c\, \e^{-r_0
t}\|\nn\log f\|_\8\|\xi-\eta\|_r
\end{equation}
holds for $\xi,\eta\in\C_r$ and  $f\in  \B_b^+(\C_r)$ with
$\|\nn\log f\|_\8<\8$. Consequently, all assertions in Theorem
\ref{T2.1} hold for $E=\C_r, \rr(\xi,\eta)= \|\xi-\eta\|_r$,
$\LL=c$, $\GG_t= c\,\e^{-r_0t},$ and $\Phi(\xi,\eta)=
c\,\|\xi-\eta\|_r^2$.
\end{thm}

To prove \eqref{eq6}, we construct coupling by change of measures
(see for example \cite{Wbook}).
 Since the memory is infinite, we cannot make the coupling successful at a fixed time,
 but can
make two marginal processes close to each other exponentially fast
when $t\to\infty$. The following construction of coupling is due to
\cite{W11}, where SDEs without memory are concerned.

We simply denote $X_t=X_t^\xi$ and $X(t)= X^\xi(t)$, the functional
solution and the solution to \eqref{eq0} with the initial value
$\xi\in \C_r$, respectively. For any $\ll>r,$ where $r>0$ is given
in \eqref{a1}, consider the following SDE:
\begin{equation}\label{eq00}
\d Y(t)=\{b(Y_t)+\ll\si(Y_t)\si^{-1}(X_t)(X(t)-Y(t))\}\d
t+\si(Y_t)\d W(t),~~~~t>0,~~~~Y_0=\eta\in\C_r.
\end{equation}
With ({\bf H1})-({\bf H3}) in hand,  we infer that ({\bf D1}) and
({\bf D2}) in  Appendix \ref{app} below hold   for
$$\tt b(\zeta):=b(\zeta)+\ll\si(\zeta)\si^{-1}(\tt \zeta)(\tt\zeta(0)-\zeta(0)),~~~\zeta\in\C_r$$
with fixed $\tt\zeta\in\C_r.$ Thus, under
 ({\bf H1})-({\bf H3}), Theorem \ref{existence} shows that \eqref{eq00} has a unique strong solution
$(Y(t))_{t\ge0}.$ Let $Y_t$ be the segment process. To examine that
$Y_t$ has the semigroup $P_t$ under a   probability measure $\Q$,
let
\begin{equation*}
h(t)=\ll\,\si^{-1}(X_t)(X(t)-Y(t)),\ \ \ \  \tt W(t)=W(t)+\int_0^th(s)\d s,
\end{equation*}
and define \beq\label{eq333} R(t)=\exp\bigg(-\int_0^t\<h(s),\d
W(s)\>-\ff{1}{2}\int_0^t|h(s)|^2\d s\bigg),\ \ t\ge 0.
\end{equation} We have the following result.

\begin{lem}\label{Gir} Assume
{\rm   ({\bf H1})-({\bf H3})}.  Then, \beq\label{UNF}
\sup_{t\in[0,T]}\E\Big(R(t)\log R(t)\Big)<\8,\ \ T>0.
\end{equation}
 Consequently, there exists a unique probability measure $\Q$ on $(\OO,\F_\infty)$ such that
\beq\label{KLM} \ff{\d \Q|_{\F_t}}{\d \P|_{\F_t}}= R(t),\ \ t\ge
0.\end{equation}   Moreover, $\tt W(t)$ is a $d$-dimensional
Brownian motion under  $\Q$.
\end{lem}

\begin{proof} If \eqref{UNF} holds, then $(R(t))_{t\ge 0}$ is a locally uniformly integrable martingale, and, by Girsanov's theorem, for any
$T>0$, $(\tt W(t))_{t\in [0,T]}$ is a $d$-dimensional Brownian motion under the probability $\Q_T:= R(T)\P.$
 By the martingale property of $R(t)$, the family $(\Q_T)_{T>0}$ is harmonic, so that by  Kolmogorov's harmonic theorem, there exists a unique probability measure $\Q$ on $(\OO,\F)$ such that \eqref{KLM} holds. Therefore, $(\tt W(t))_{t\ge 0}$ is a   $d$-dimensional Brownian motion under  $\Q$. So, it remains to prove \eqref{UNF}.

 For any $k>\|\xi\|_r+\|\eta\|_r$, define the stopping
time
\begin{equation*}
\tau_k=\inf\{t\ge0:\|X_t\|_r+\|Y_t\|_r\ge k\}.
\end{equation*}
Due to the non-explosion of \eqref{eq0} and \eqref{eq00} (see
\eqref{e9} below for more details), $\tau_k\uparrow\8$ as
$k\uparrow\8.$ Then, $(\tt W(t))_{t\in[0,T\wedge\tau_k]}$ is a
Brownian motion under the weighted probability measure
$\d\mathbb{Q}_{T,k}=R(T\wedge\tau_k)\d \P$. By ({\bf H3}), there
exists a constant $c_1>0$ such that
\begin{equation}\label{02}
\begin{split}
\E\Big(R(t\wedge\tau_k)\log R(t\wedge\tau_k)\Big)&=\E_{\mathbb{Q}_{T,k}}\log R(t\wedge\tau_k)
 =\ff{1}{2}\E_{\mathbb{Q}_{T,k}}\int_0^{t\wedge\tau_k}|h(s)|^2\d s\\
&\le c_1\int_0^T\E_{\mathbb{Q}_{T,k}}\|X_{t\wedge\tau_k}-Y_{t\wedge\tau_k}\|^2_r\d t,~~~~~t\in[0,T].
\end{split}
\end{equation}  Rewrite  respectively  \eqref{eq0} and \eqref{eq00} as
\beq\label{RWT} \beg{split} &\d  X(t)  = \big\{b(X_t)-\ll (X(t)-Y(t))\big\}\d t + \si(X_t)\d\tt W(t),\ \ t\le T\land\tau_k,~X_0=\xi\\
 &\d Y(t)= b(Y_t)\d t +\si(Y_t)\d\tt W(t),\ \  t\le T\land\tau_k,~Y_0=\eta.\end{split}\end{equation}
 By It\^o's formula and assumptions ({\bf H1}) and ({\bf H2}), under the probability $\Q$ we have
 $$\d |X(t)-Y(t)|^2 \le c_2\|X_t-Y_t\|_r^2 \d t+ 2\<X(t)-Y(t), (\si(X_t)-\si(Y_t))\d\tt W(t)\>, \ \
 t\le T\land \tau_k$$
 for some constant $c_2>0.$
 Therefore, applying the BDG inequality and using ({\bf H2}) once more, we can find out a constant $C(T,\xi,\eta)>0$ such that
 $$\E\|X_{t\land \tau_k}-Y_{t\land \tau_k}\|_r^2 \le C(T,\xi,\eta),\ \ t\in [0, T].$$ Plugging this into \eqref{02} leads to
 $$\sup_{k\ge 0, t\in [0,T] } \E\Big(R(t\land \tau_k)\log R(t\land\tau_k)\Big) <\infty,\ \ T>0.$$
 This implies \eqref{UNF} by Fatou's lemma.\end{proof}

Next, to deduce  asymptotic log-Harnack inequality from the
asymptotic coupling $(X_t,Y_t)$,  we show that  $\|X_t-Y_t\|_r$
decays exponentially fast as $t\to\infty$ in the $L^p$-norm sense.

\begin{lem}\label{lem} Assume
{\rm   ({\bf H1})-({\bf H3\textbf{}})}. Then, for any $p >0$ and
$r_0\in (0,r)$, there exist $\ll,c>0$ such that the above asymptotic
coupling $(X_t,Y_t)$ satisfies
\begin{equation}\label{b3}
\E_\Q\|X_t-Y_t\|_r^p\le c\,\e^{-p\,r_0t}\|\xi-\eta\|_r^p,\ \ t\ge 0.
\end{equation}
\end{lem}

\begin{proof} By Jensen's inequality, it suffices to prove for large $p>0$, for instance, $p>4$ as we will take below.

 Let $Z(t)=X(t)-Y(t), t\in\R$. According to Lemma \ref{Gir}, \eqref{RWT} holds for all $t\ge 0$, where $\tt W(t) $ is a $d$-dimensional Brownian motion
  under the probability measure $\Q$.
By applying It\^o's formula and using ({\bf H1}) and ({\bf H2}),
there exists     $K>0$ such that for all $\ll>r$,
\begin{equation}\label{bao}\d |Z(t)|^2 \le  \{-2\ll|Z(t)|^2+ K \|Z_t\|_r^2\} \d t +
2\<Z(t), (\si(X_t)-\si(Y_t))\d\tt W(t)\>, \ \
 t\ge 0.\end{equation} Set
 \begin{equation*}\label{LLM}    M^{(\ll)}(t):=2\,\int_0^t \e^{2\ll s}\<Z(s),(\si(X_s)-\si(Y_s))\d \tt W(s)\>,\ \ t\ge 0.\end{equation*}  Thus, we deduce
 from \eqref{bao} and the It\^o formula that 
\begin{equation}\label{SVV}
\e^{2\ll t}|Z(t)|^2\le|Z(0)|^2+K\int_0^t\e^{2\ll s}\|Z_s\|_r^2\d
s+M^{(\ll)}(t),\ \ t\ge 0.
\end{equation}
So, letting   $\kk =2(\ll-r)>0$, we obtain
\begin{equation}\label{b6}
\begin{split}
\e^{2rt}|Z(t)|^2&\le\e^{-\kk t}|Z(0)|^2+K\int_0^t\e^{-\kk(t-
s)}\e^{2rs}\|Z_s\|_r^2\d s + \e^{-\kk t} M^{(\ll)}(t),\ \ t\ge 0.
\end{split}
\end{equation}
Combining this with the fact that
\begin{equation}\label{a8}
\begin{split}
\|Z_t\|_r^2&=\sup_{-\8<\theta\le0}(\e^{2r\theta}|Z(t+\theta)|^2)
 = \e^{-2rt}\sup_{-\8<s\le t}(\e^{2r s}|Z(s)|^2)\\
&\le\e^{-2rt}\|Z_0\|_r^2+\e^{-2rt}\sup_{0\le s\le t}(\e^{2r
s}|Z(s)|^2),
\end{split}
\end{equation}
we  arrive at
\begin{equation}\label{b1}\beg{split}
\e^{2r t}\|Z_t\|_r^2&\le\|Z_0\|_r^2+\sup_{0\le s\le t}(\e^{2r
s}|Z(s)|^2)\\
&\le 2\|Z_0\|_r^2 + K \int_0^t\e^{-\kk(t- s)}\e^{2rs}\|Z_s\|_r^2\d s
+\sup_{s\in [0,t]} \Big(\e^{-\kk s} M^{(\ll)}(s)\Big),\end{split}
\end{equation}
where in the second procedure we have utilized
$t\mapsto\e^{2rt}\|Z_t\|_r^2$ is nondecreasing. By H\"older's
inequality, one finds that
\begin{equation}\label{y2}
\begin{split}
\Big(\int_0^{t\wedge\tau_k}\e^{-\kk(t\wedge\tau_k-
s)}\e^{2rs}\|Z_s\|_r^2\d s\Big)^{p/2}&\le\Big(\int_0^\8\e^{-\ff{p\kk
s}{p-2}}\d
s\Big)^{\ff{p-2}{2}}\int_0^{t\wedge\tau_k}\e^{prs}\|Z_s\|_r^p\d
s\\
&\le\Big(\ff{p-2}{p\kk}\Big)^{\ff{p-2}{2}}\int_0^{t\wedge\tau_k}\e^{prs}\|Z_s\|_r^p\d
s.
\end{split}
\end{equation}
On the other hand, taking advantage of \cite[Lemma 2.2]{E09}, we
may find out a constant $c_0(p,\ll)>0$ with
$\lim_{\ll\to\8}c_0(p,\ll)=0$ such that
\begin{equation}\label{y1}
\begin{split}
&\E_\Q\bigg(\sup_{0\le s\le
t\wedge\tau_k}\Big(\e^{-\kk s} M^{(\ll)}(s)\Big)^{p/2}\bigg)\\
&\le  c_0(p,\ll)\E_\Q\int_0^{t\wedge\tau_k}\e^{pr
s}|(\si(X_s)-\si(Y_s))^*Z(s)|^{p/2}\d s\\
&\le
K_2^{p/4}c_0(p,\ll)\E_\Q\int_0^{t\wedge\tau_k}\e^{prs}\|Z_s\|_r^p\d
s\\
&\le
K_2^{p/4}c_0(p,\ll)\int_0^t\E_\Q(\e^{pr(s\wedge\tau_k)}\|Z_{s\wedge\tau_k}\|_r^p)\d
s.
\end{split}
\end{equation}
Taking \eqref{y2} and \eqref{y1} into consideration,   we deduce
from \eqref{b1} that, for some $c(p)$, $c(p,\ll)\in (0,\infty)$ with
$c(p,\ll)\downarrow 0$ as $\ll\uparrow \infty$,
$$\E_\Q (\e^{p\,r (t\land \tau_k)}\|Z_{t\land\tau_k}\|_r^{p})\le c(p) \|Z_0\|_r^p + c(p,\ll) \int_0^t\E_\Q (\e^{p\,r(s\land \tau_k)}\|Z_{s\land\tau_k}\|_r^p)\d s,\ \ t\ge 0.$$
  By
Gronwall's lemma, it follows that
$$\E_\Q (\e^{p\,r (t\land \tau_k)}\|Z_{t\land\tau_k}\|_r^{p})\le c(p) \e^{c(p,\ll)t} \|Z_0\|_r^p,\ \ t\ge 0.$$
Letting $k\to\infty$,  we obtain from Fatou's lemma that
$$\E_\Q \|Z_{t}\|_r^{p}\le c(p) \e^{-(p\,r-c(p,\ll))t} \|Z_0\|_r^p,\ \ t\ge 0,$$
which yields the desired assertion due to   $c(p,\ll)\to 0$ as
$\ll\to\infty$.
\end{proof}

\begin{proof}[Proof of Theorem \ref{th1}] By Lemma \ref{Gir} and the weak uniqueness of   solutions to   \eqref{eq0}, we have
$$P_tf(\eta)= \E_\Q f(Y_t),\ \ t\ge 0, \, f\in \B_b(\C_r).$$
 So, for any $f\in \B_b^+(\C_r)$ with $\|\nn\log f\|_\infty<\8$,  by the definition of $\|\nn\log f\|_\infty$ and Lemma \ref{lem}, there exists a constant $C>0$ such that
\begin{equation}\label{eq9}
\begin{split}
P_t\log f(\eta)&=\E_\mathbb{Q}\log f(Y_t)
 =\E_\mathbb{Q}\log f(X_t)+\E_\mathbb{Q}(\log f(Y_t)-\log f(X_t))\\
&\le\E(R(t)\log f(X_t))+\|\nn\log
f\|_\8\E_\mathbb{Q}\|X_t-Y_t\|_r\\
&\le\E(R(t)\log R(t))+\log P_t f(\xi)+C\e^{-r_0t}\|\nn\log
f\|_\8  \|\xi-\eta\|_r,
\end{split}
\end{equation}
where in the last display we have used   the Young inequality; see
e.g. \cite[Lemma 2.4]{ATW}.

Next,  it follows from  \eqref{eq333}, \eqref{b3}, ({\bf H2}) and
({\bf H3}) that for some constants $C_1,C_2>0$, \beg{align*}
&\E(R(t)\log R(t)) =\E_\mathbb{Q}\log R(t)
 =\ff{\ll^2}{2}\E_\mathbb{Q}\int_0^t|\si^{-1}(X_t)(X(t)-Y(t))|^2\d t\\
&\le\ff{C_1\,\ll^2}{2}\int_0^t\E_\mathbb{Q}|X(s)-Y(s)|^2\d s
 \le\ff{C_1\,\ll^2}{2}\int_0^t\E_\mathbb{Q}\|X_s-Y_s\|^2_r\d s
 \le C_2\ll^2 \|\xi-\eta\|_r^2.\end{align*}
 Plugging this back into \eqref{eq9} yields \eqref{eq6}.
\end{proof}

\section{Neutral SDEs of infinite memory}

Consider the following  neutral type SDEs   with infinite memory:
\begin{equation}\label{a2}
\d \{X(t)-G(X_t)\}=b(X_t)\d t+\si(X_t)\d
W(t),~~~~t>0,~~~~X_0=\xi\in\C_r,
\end{equation}
where $b,\si$ and $W$ are stipulated as in \eqref{eq0}, and
$G:\C_r\to\R^d$, which is, in general, named as the neutral term of
\eqref{a2}. This kind of equation has been utilized to model some
evolution phenomena arising in, e.g., physics, biology and
engineering, to name a few; see  e.g.  \cite{M08}. Besides ({\bf
H2}) and ({\bf H3}) above, we further assume that
\begin{enumerate}
\item[({\bf A1})] There exists $\dd\in(0,1)$ such that
$|G(\xi)-G(\eta)|\le\dd\|\xi-\eta\|_r$ for any $\xi,\eta\in\C_r;$

\item[({\bf A2})] $b\in C(\C_r)$ is  bounded on bounded subsets of $\C_r$ and there exists an $L>0$ such that
\begin{equation*}
2\<\xi(0)-\eta(0)-(G(\xi)-G(\eta)),b(\xi)-b(\eta)\>\le
L\|\xi-\eta\|_r^2,~~~~\xi,\eta\in\C_r.
\end{equation*}
\end{enumerate}

Under assumptions ({\bf A1}), ({\bf A2}) and ({\bf H2}), \eqref{a2}
has a unique strong solution $(X^\xi(t))_{t\ge0}$  with the initial
value $\xi\in \C_r$ by following exactly the argument of Theorem
\ref{existence} below. Let $(X^\xi_t)_{t\ge0}$ be the corresponding
segment process.

\begin{thm}\label{th3} Assume
{\rm   ({\bf H2})-({\bf H3}) and ({\bf A1})-({\bf A2})}. Then all assertions in Theorem \ref{th1} hold true.
\end{thm}

\begin{proof} As in the proof of Theorem \ref{th1}, we construct the following asymptotic coupling by change of measures.
Write $(X(t),X_t)=(X^\xi(t),X^\xi_t)$ for notation brevity and
consider the coupled neutral SDE with $Y_0=\eta$:
\begin{equation*}
\d
\{Y(t)-G(Y_t)\}=\{b(Y_t)+\ll\si(Y_t)\si^{-1}(X_t)(X(t)-Y(t)-(G(X_t)-G(Y_t))\}\d
t+\si(Y_t)\d W(t).
\end{equation*}
  For any $t\ge0$
and $\ll>r$, let
\begin{equation*}
h(t):=\ll\,\si^{-1}(X_t)\{X(t)-Y(t)-(G(X_t)-G(Y_t))\},\ \ \tt
W(t):=W(t)+\int_0^th(s)\d s.
\end{equation*}
Define
\begin{equation}\label{s1}
R(t)=\exp\Big(-\int_0^t\<h(s),\d W(s)\>-\ff{1}{2}\int_0^t|h(s)|^2\d
s\Big),~~~t\ge0.
\end{equation}
By a close inspection of argument for  Theorem \ref{th1}, it
suffices to prove \eqref{UNF} and
  Lemma \ref{lem} for the present asymptotic coupling $(X(t),Y(t))$. Below, we merely present a brief proof for the later since the former one can
   be done as that of  Lemma \ref{Gir}.  Let   the probability measure $\Q$  be given by \eqref{KLM} with $R(t)$ defined in \eqref{s1}.
Then $\tt W(t)$ is a $d$-dimensional Brownian motion under $\Q$.
Again let $Z(t)=X(t)-Y(t), t\in\R$, and $Z_t$ be the associated
segment process.  By following the
  argument   to derive \eqref{b6},     ({\bf H2})  and ({\bf A2}) imply
\begin{equation}\label{c44}
\begin{split}
 \e^{2rt}|Z(t)-(G(X_t)-G(Y_t))|^2&\le \e^{-\kk t}|Z(0)-(G(X_0)-G(Y_0))|^2\\
&\quad+K\int_0^t\e^{-\kk(t- s)}\e^{2rs}\|Z_s\|_r^2\d
s+\e^{-\kk t}M^{(\ll)}(t)
\end{split}
\end{equation} for $\kk:=2(\ll-r)>0$, some constant $K>0$ and
$$M^{(\ll)}(t):=2\int_0^t \e^{2\ll s}\<Z(s)-(G(X_s)-G(Y_s)),(\si(X_s)-\si(Y_s))\d
\tt W(s)\>.$$ Next, for any  $\vv>0,$ it follows from ({\bf A1})
that
\begin{equation*}
\begin{split}
|Z(t)|^2&\le
(1+\vv)|Z(t)-(G(X_t)-G(Y_t))|^2+(1+1/\vv)|G(X_t)-G(Y_t)|^2\\
&\le(1+\vv)|Z(t)-(G(X_t)-G(Y_t))|^2+(1+1/\vv)\dd^2\|Z_t\|^2_r.
\end{split}
\end{equation*}
This,  together with \eqref{b1}, yields
\begin{equation*}
\begin{split}
\e^{2rt}\|Z_t\|_r^2 &\le\|Z_0\|_r^2+(1+\vv) \sup_{0\le s\le
t}(\e^{2rs}|Z(s)-(G(X_s)-G(Y_s))|^2)+(1+1/\vv)\dd^2\e^{2rt}\|Z_t\|_r^2.
\end{split}
\end{equation*}
Taking $\vv=\ff{\dd}{1-\dd}$, we derive
\begin{equation*}
\begin{split}
\e^{2rt}\|Z_t\|_r^2
&\le\ff{1}{1-\dd}\Big\{\|Z_0\|_r^2+\ff{1}{1-\dd}\sup_{0\le s\le
t}(\e^{2rs}|Z(s)-(G(X_s)-G(Y_s))|^2)\Big\}.
\end{split}
\end{equation*}
Combining this  with \eqref{c44}, and noting that ({\bf A1}) implies
\begin{equation*}\label{a7}
|\xi(0)-\eta(0)-(G(\xi)-G(\eta))|^2\le4\|\xi-\eta\|^2_r,~~~~\xi,\eta\in\C_r,
\end{equation*}
we arrive at
\begin{equation}\label{a6}
\begin{split}
 \e^{prt}\|Z_t\|_r^p
&\le
c\,\Big\{\|Z_0\|_r^p+\Big(\int_0^{t\wedge\tau_k}\e^{-\kk(t\wedge\tau_k-
s)}\e^{2rs}\|Z_s\|_r^2\d s\Big)^{p/2}\\
&\quad+\sup_{0\le s\le
t\wedge\tau_k}\Big(\e^{-\kk s} M^{(\ll)}(s)\Big)^{p/2}\Big\}\\
\end{split}
\end{equation} for some constant  $c>0$. With the aid of ({\bf A1}) and ({\bf H3}), we observe that \eqref{y2} and \eqref{y1} still hold  for $p>4$. Combining this with
  $\kk\to\infty$ as $\ll\to\infty$,   we deduce
\eqref{b3} from \eqref{a6}.
\end{proof}

\begin{rem}
{\rm Indeed, the existence and uniqueness of solutions to \eqref{a2}
under the locally weak monotone condition and the weak coercive
condition can be obtained by following the argument of Theorem
\ref{app} and constructing the following Euler-Maruyama scheme
\begin{equation*}
\d \{X^n(t)-G(X_t^n)\}=b(\hat X_t^n)\d t+\si(\hat X_t^n)\d
W(t),~~~t>0,~~~X_0^n=X_0=\xi,
\end{equation*}
where, for  $t\ge0$, $ \hat X_t^n(\theta):=X^n((t+\theta)\wedge
t_n),~ t_n:= [nt]/n,\ \theta\in(-\8,0]. $ }
\end{rem}

\section{Semi-linear  SPDEs of infinite memory}

Let $(\H,\langle\cdot,\cdot\rangle,|\cdot|)$ be a real separable
Hilbert space. $\C=C((-\8,0];\H)$ denotes the family of all
continuous mappings $f:(-\8,0]\rightarrow\H$, and $\C_r$ is defined
as in \eqref{a1}. Let $\scr {L}(\H)$ and $\scr {L}_{HS}(\H)$ be the
spaces of all bounded linear operators and Hilbert-Schmidt operators
on $\H$, respectively. Denote $\|\cdot\|$ and $\|\cdot\|_{HS}$ by
the operator norm and the Hilbert-Schmidt norm, respectively.

Consider the following semi-linear SPDE on $\H$ with infinite
memory:
\begin{equation}\label{w0}
\d X(t)=\{AX(t)+b(X_t)\}\d t+\si(X_t)\d W(t),~~~t>0,~~~~X_0=\xi,
\end{equation}
where $(A,\D(A))$ is a densely defined closed operator on $\H$
generating a $C_0$-semigroup $\e^{tA}$, $b:\C_r\rightarrow \H$,
$\si:\C_r\rightarrow\scr {L}(\H)$, and $(W(t))_{t\geq0}$ is  a
cylindrical Wiener process on $\H$ for  a complete probability space
$(\Omega, \scr {F}, \mathbb{P})$ with the natural filtration $(\scr
{F}_t)_{t\geq0}$.

 We
assume that
\begin{enumerate}
\item[{\bf (B1)}] ($-A, \mathscr{D}(A))$  is   self-adjoint  with
  discrete spectrum
 $0<\ll_1\le \ll_2\le \cdots $ counting  multiplicities such that
 $\sum\limits_{i\ge 1}  \ll_i^{-\aa}<\infty$ for some $\aa\in (0,1)$;

\item[{\bf (B2)}] There exists an $L_0>0$  such that
\begin{equation*}
|b(\xi)-b(\eta)|+\|\si(\xi)-\si(\eta)\|_{HS}\le
L_0\|\xi-\eta\|_r,~~~\xi,\eta\in\C_r;
\end{equation*}

\item[{\bf (B3)}]  $\|\si\|_\8:=\sup\limits_{\xi\in\C_r}\|\si(\xi)\|<\8$, and $\si(\xi)$ is invertible with
 $\|\si^{-1}\|_\8:=\sup\limits_{\xi\in\C_r}\|\si^{-1}(\xi)\|<\8$.
\end{enumerate}
According to {\bf (B2)} and {\bf (B3)},    $\si$ need not,  but the
difference $\si(\xi)-\si(\eta)$ does, take values in the space of
Hilbert-Schmidt operators.  Recall that
 a continuous adapted process
$(X_t^\xi)_{t\ge0}$ on $\C_r$ is called a mild solution to
\eqref{w0} with the initial value $\xi\in \C_r$, if  $X^\xi_0=\xi$
and
\begin{equation*}
X^\xi(t)=\e^{tA}\xi(0)+\int_0^t\e^{(t-s)A}b(X_s^\xi)\d
s+\int_0^t\e^{(t-s)A}\si(X_s^\xi)\d W(s),~~~~t\ge 0.
\end{equation*}
In terms of  the following result, assumptions  {\bf(B1)-(B3)} imply
the existence and uniqueness of mild solutions to \eqref{w0} as well
as asymptotic log-Harnack inequality of the associated Markov
semigroup.

\begin{thm}\label{spde} Assume
   {\bf(B1)-(B3)}.  Then $\eqref{w0}$ has  a unique mild solution  $(X_t^\xi)_{t\ge0}$, and the associated Markov semigroup $P_t$ satisfies all assertions in Theorem \ref{th1}.
\end{thm}

\begin{proof} {\bf (a)} The   existence and uniqueness of mild solutions follows from  the Banach fixed point
theorem by a more or less standard argument under the assumptions
{\bf(B1)-(B3)}. Fix $T>0$ and let \begin{equation*}
\begin{split}\D_T=\Big\{ &(u(t))_{t\in (-\8,T]} \ \text{is\ a\
continuous\ adapted \ process\ on\ } \H
\mbox{ with } u_0=\xi  \\
&\mbox{ and } \E\Big(\sup_{t\in (-\8,T]} (\e^{rt}|u(t)|^4)\Big)
<\infty\Big\}.\end{split}\end{equation*}
 Then $\D_T$ is a complete metric space with
$$\rr(u,v):=\|u-v\|_{\D_T}:=  \Big(\E\Big(\sup_{t\in [0,T]} (\e^{r t}|u(t)-v(t)|^4)\Big)\Big)^{\ff 1 4}.$$
Observe that the metric   $\rr$  is equivalent to the metric below
$$\rr_0(u,v):=\|u-v\|_{\D_T^0}:=  \Big(\E\Big(\sup_{t\in [0,T]}  |u(t)-v(t)|^4) \Big)^{\ff 1 4}.$$
By {\bf (B1)}-{\bf (B3)}, it is easy to see that
\begin{equation}\label{p1}
\GG(u)(t):=\e^{tA}\xi(0)+\int_0^t\e^{(t-s)A}b(u_s)\d
s+\int_0^t\e^{(t-s)A}\si(u_s)\d W(s),~~ t\ge 0,~~u\in\D_T
\end{equation} gives rise to a map from $\D_T$ to $\D_T$. Then, by virtue of the fixed point theorem,
it remains to find a constant $T_0>0$ independent of
 $\xi$ such that, for any $T\le T_0$, the map
$\GG$ is contractive in $\D_T$ sine the existence and uniqueness of
mild solution on the intervals $[T_0,2T_0], [2T_0,3T_0], \cdots$ can
be done inductively. Below we provide a brief proof for this.

For any $u,v\in\D_T$, by \eqref{p1} we have
$$\d \{\GG(u)(t)-\GG(v)(t)\} =\big\{A\big(\GG(u)(t)-\GG(v)(t)\big) + b(u_t)-b(v_t)\big\}\d t +\{\si(u_t)-\si(v_t)\}\d W(t).$$
According to  {\bf (B1)}-{\bf (B3)}, we may apply It\^o's formula to
$|\GG(u)(t)-\GG(v)(t)|^2$ to derive that there exists $c_1>0$ such
that
\beg{align*}\d |\GG(u)(t)-\GG(v)(t)|^2&=   2\big\<  \GG(u)(t)-\GG(v)(t), A\big(\GG(u)(t)-\GG(v)(t)\big)+b(u_t)-b(v_t)\big\>\d t\\
& \quad +\|\si(u_t)-\si(v_t)\|_{HS}^2 \d t +\d M(t)\\
&\le\ff{1}{2\ss{3}T} |\GG(u)(t)-\GG(v)(t)|^2\d t+c_1(1+
T)\|u_t-v_t\|_r^2\d t +\d M(t),\end{align*} where we have used the
negative definite property of $A$ due to {\bf (B1)} in the last step
and set
$$M(t):= 2 \int_0^t \big\< \GG(u)(s)-\GG(v)(s), \{\si(u_s)-\si(v_s)\}\d W(s)\big\>$$
is a martingale. By the BDG inequality, there exists a constant $c_2>0$ such that
\begin{equation}\label{s2}
\begin{split}  &\|\GG(u)-\GG(v)\|_{\D_T^0}^4 =\E\Big(\sup_{t\in [0,T]} |\GG(u)(t)-\GG(v)(t)|^4\Big) \\
&\le\ff{1}{4}\|\GG(u)-\GG(v)\|_{\D_T^0}^4 + 3c_1^2\,(1+ T)^2T^2
\|u-v\|_{\D_T}^4+3\,
\E\Big(\sup_{t\in [0,T]}M(t)^2\Big)\\
&\le \ff{1}{4}\|\GG(u)-\GG(v)\|_{\D_T^0}^4 + 3c_1^2(1+ T)^2T^2
\|u-v\|_{\D_T}^4+ c_2\, \E \<M (T)\>.
\end{split}
\end{equation} Note that the
definition of $M(t)$ and the assumption   {\bf(B2)} imply
\beg{align*}  \E \<M (T)\> &\le 4L_0 \int_0^T \E (\GG(u)(t)-\GG(v)(t)|^2\|u_t-v_t\|_r^2)\d t\\
&\le 4L_0 T   \|\GG(u)-\GG(v)\|_{\D_T^0}^2\|u-v\|_{\D_T}^2\\
&\le \ff 1 {4c_2} \|\GG(u)-\GG(v)\|_{\D_T^0}^4 + 16 c_2 L_0^2 T^2
\|u-v\|_{\D_T}^4.\end{align*} Putting this into \eqref{s2} gives
that
\begin{equation*}
\|\GG(u)-\GG(v)\|_{\D_T^0}^4\le2(3c_1^2(1+T^2)+16c_2^2L_0^2)T^2\|u-v\|_{\D_T}^4
\end{equation*}
so that
$$\|\GG(u)-\GG(v)\|_{\D_T}^4 \le 2(3c_1^2(1+T^2)+16c_2^2L_0^2)T^2\e^{4rT}\|u-v\|_{\D_T}^4.$$ Therefore,
by taking $T_0>0$ such that
$2(3c_1^2(1+T^2_0)+16c_2^2L_0^2)T^2_0\e^{4rT_0}<1$, we conclude that
$\GG$ is contractive in $\D_T$ for any  $T\le T_0.$

{\bf (b)} It remains to verify the asymptotic log-Harnack inequality
\eqref{eq6}. To this end, we construct  an asymptotic coupling by
change of
 measures as follows. Let $(X(t),X_t)= (X^\xi(t), X_t^\xi)$, and  for any $\ll>0$, consider the following  SPDE with $Y_0=\eta$:
\begin{equation}\label{p3}
\d Y(t)=\{AY(t)+b(Y_t)+\ll\si(Y_t)\si^{-1}(X_t)(X(t)-Y(t))\}\d
t+\si(Y_t)\d W(t),~~t>0.
\end{equation}  As shown in {\bf (a)}, assumptions ({\bf B1})-({\bf
B3}) imply that \eqref{p3} has a unique local mild solution
$(Y(t))_{t\ge 0}$. Moreover, since  the drift $\tt
b(\zeta):=b(\zeta)+\ll\si(\zeta)\si^{-1}(\tt\zeta)(\tt\zeta(0)-\zeta(0))$
for any $\zeta\in\C_r$ and fixed $\tt\zeta\in\C_r$ is of linear
growth due to ({\bf B2}) and ({\bf B3}), we indeed deduce that the
unique local mild solution is the global one. Let $(Y_t)_{t\ge0}$ be
the associated segment process. For any $t\ge0$ and $\ll>r$, set
\begin{equation*}
h(t):=\ll\,\si^{-1}(X_t)(X(t)-Y(t)),\ \ \  \tt W(t):=W(t)+\int_0^th(s)\d s.
\end{equation*}
Define
\begin{equation*}
R(t)=\exp\bigg(-\int_0^t\<h(s),\d W(s)\>-\ff{1}{2}\int_0^t|h(s)|^2\d
s\bigg).
\end{equation*} As explained in the proof of Theorem \ref{th3},   we only need to verify \eqref{UNF} and \eqref{b3} for the present framework. By a standard finite-dimensional approximation argument
(see for instance \cite[Theorem 4.1.3]{Wbook}), these can be easily deduced from    assumptions   ({\bf B1})-({\bf B3}). We therefore skip the details to save space.  \end{proof}

\section{Stochastic Hamiltonian systems of infinite memory}

In this section, we establish the asymptotic log-Harnack inequality
for a class of degenerate SDEs of infinite memory. More precisely,
we consider the following  stochastic Hamiltonian system of infinite
memory on $\R^{2d}:=\R^d\times\R^d$
\begin{equation}\label{d1}
\begin{cases}
\d X(t)=\ll\, Y(t)\d t\\
\d Y(t)=b(X_t,Y_t)\d t+\si(X_t,Y_t)\d W(t)
\end{cases}
\end{equation}
with the initial value $(X_0,Y_0)=(\xi,\eta)\in\C_r\times\C_r,$
where $\ll>0,$ $b:\C_r\times\C_r\rightarrow\R^d$,
$\si:\C_r\times\C_r\rightarrow\R^d\otimes\R^d$, and $(W(t))_{t\ge0}$
is a  $d$-dimensional Brownian motion defined on the  probability
space $(\OO,\F,(\F_t)_{t\ge0},\P)$. When the memory is finite or
empty, this model has been intensively investigated, see for
instance
 \cite{BWYb,MSH,W14,WZ,WZ1,Zhang} for results on derivative formulas, Harnack inequalities, hypercontractivity, ergodicity,   well-posedness, and so forth.

  To investigate the present setup with infinite memory, we make the following assumptions.

\begin{enumerate}
\item[({\bf C1})] There exist constants  $\bb,L_1>0$ such that  for any $(\xi,\eta), (\bar \xi,\bar
\eta)\in\C_r\times\C_r,$
\begin{equation*}
\<\bb(\xi(0)-\bar \xi(0))+(\eta(0)-\bar \eta(0)),b(\xi,\eta)-b(\bar
\xi,\bar \eta)\> \le L_1(\|\xi-\bar \xi\|^2_r+\|\eta-\bar
\eta\|^2_r);
\end{equation*}
\item[({\bf C2})]There exists an $L_2>0$ such that  for any $(\xi,\eta), (\bar \xi,\bar
\eta)\in\C_r\times\C_r,$
\begin{equation*}
\|\si(\xi,\eta)-\si(\bar\xi,\bar\eta)\|_{HS}^2\le L_2(\|\xi-\bar
\xi\|^2_r+\|\eta-\bar \eta\|^2_r);
\end{equation*}

\item[({\bf C3})]  $\|\si\|_\8:=\sup\limits_{\xi,\eta\in\C_r}\|\si(\xi,\eta)\|<\8$, and  $\si$ is invertible
with $\sup\limits_{\xi,\eta\in\C_r}\|\si^{-1}(\xi,\eta)\|<\8$.
\end{enumerate}

Under   assumptions ({\bf C1}) and ({\bf C2}), \eqref{d1} admits
from Theorem \ref{existence} in the Appendix \ref{app} a unique
strong solution $(X^\xi(t),Y^\eta(t))_{t\ge0}$ with the
corresponding segment process $(X_t^\xi,Y_t^\eta)_{t\ge0}$, which is
a homogeneous Markov process. Let $P_t$ be the semigroup generated
by $(X_t^\xi,Y_t^\eta)$, i.e., $P_tf(\xi,\eta)=\E
f(X_t^\xi,Y_t^\eta)$, $f\in\B_b(\C_r\times\C_r)$.

For $p>2$ and $ \ff{1}{p}<\aa<\ff{1}{2}$, let
\begin{equation*}
\Lambda_{p,\aa}=\bigg(\ff{p^{1+p}}{2(p-1)^{p-1}}\bigg)^{p/2}\bigg(\ff{\Gamma(1-2\aa)}{2^{1-2\aa}}\bigg)^{p/2}\bigg(1-\ff{1}{p}\bigg)^{p\aa-1}
\Gamma\Big(\ff{p\aa-1}{p-1}\Big)^{p-1},
\end{equation*}
  where $\Gamma(\cdot)$ is the Gamma function, 
  it is easy to show that 
  the function $(p,\aa)\mapsto\Lambda_{p,\aa}$ achieves its
infimum at some point $(p_0,\aa_0)$, i.e. 
\begin{equation*}
\Lambda_{p_0,\aa_0}=\inf_{p>2,\ff{1}{p}<\aa<\ff{1}{2}}\Lambda_{p,\aa},~~~p_0>2,~~~\ff{1}{p_0}<\aa_0<\ff{1}{2}.
\end{equation*}
 Moreover, we denote
\begin{equation}\label{s8}
\mu_{p_0}=2^{3p_0-1}\Big\{(L_1+L_2/2)^{p_0}(1-1/p_0)^{p_0-1}+\Lambda_{p_0,\aa_0}L_2^{p_0/2}\Big\}.
\end{equation}

\begin{thm}\label{dege}
{\rm Assume {\bf(C1)-(C3)}.  If
\begin{equation}\label{s7}
\ll>r+\ff{1+\bb+2\bb^2}{2\bb}\Big(\ff{\mu_{p_0}}{2p_0r}\Big)^{\ff{2}{p_0-2}},
\end{equation}
then there exist $r_0\in(0,r)$ and a constant $c>0$  such that
\begin{equation}\label{p2}
\begin{split}
P_t\log f(\xi,\eta)  &\le\log P_t
f(\xi',\eta')+c\,(\|\xi-\xi'\|_r^2+\|\eta-\eta'\|_r^2)\\
&\quad+c\, \e^{-r_0 t}\|\nn\log
f\|_\8(\|\xi-\xi'\|_r+\|\eta-\eta'\|_r)
\end{split}
\end{equation}
holds for $(\xi,\eta),(\xi',\eta')\in\C_r\times\C_r$ and $f\in
\B_b^+(\C_r\times\C_r)$ with $\|\nn\log f\|_\8<\8$. Consequently,
all assertions in Theorem \ref{T2.1} hold true. }
\end{thm}

\begin{proof}
Again, we adopt the asymptotic coupling by change of measures. Let
$(X(t),Y(t))$ solve \eqref{d1} for $(X_0,Y_0)=(\xi,\eta).$  For
  $\ll> 0$ in \eqref{d1} and   $\beta>0$ in ({\bf C1}),
consider the following stochastic Hamiltonian system
\begin{equation}\label{d2}
\begin{cases}
\d \bar X(t)=\ll\,\bar Y(t)\d t\\
\d  \bar Y(t)=\Big\{b(\bar X_t,\bar Y_t)+\si(\bar X_t,
\bar Y_t)\si^{-1}(X_t,Y_t)\Big(\ll(X(t)- \bar X(t))\\
~~~~~~~~~~~~~~~+2\ll\bb(Y(t)-\bar Y(t))\Big)\Big\}\d t+\si(\bar
X_t,\bar Y_t)\d W(t)
\end{cases}
\end{equation}
with the initial value $(\bar X_0, \bar Y_0)=(\bar
\xi,\bar\eta)\in\C_r\times\C_r$. Under ({\bf C1})-({\bf C3}),
according to Theorem \ref{app}, \eqref{d2} has a unique strong
solution $(\bar X(t),\bar Y(t))_{t\ge0}$  with the associated
segment process $(\bar X_t,\bar Y_t)_{t\ge0}$. For any $t\ge0$, let
 \begin{equation*}
h(t)=\si^{-1}(X_t,Y_t)\Big(\ll(X(t)-\bar X(t))+2\ll\,\bb\,(Y(t)-\bar
Y(t))\Big),\ \ \tt W(t)= W(t)+\int_0^th(s)\d s.
\end{equation*}
Define
\begin{equation*}
R(t)=\exp\bigg(-\int_0^t\<h(s),\d W(s)\>-\ff{1}{2}\int_0^t|h(s)|^2\d
s\bigg).
\end{equation*}
As shown in the proof of Theorem \ref{th1},  it suffices to prove
\eqref{UNF} and Lemma \ref{lem} for the present coupling
$((X_t,Y_t),(\bar X_t,\bar Y_t))$. For simplicity, we only prove the
latter one. It is easy to see that  for any $x,y\in\R^d$,
\beq\label{*SV}\ff{1}{4}\,(|x|^2+|y|^2)\le V(x,y):=
(1/2+\bb^2)|x|^2+|y|^2/2+\bb\<x,y\>\le c_\bb(|x|^2+|y|^2),
\end{equation}
where $c_\bb:=(1+\bb+2\bb^2)/2$.
 Set $Z(t):=(X(t)-\bar X(t),
Y(t)-\bar Y(t)),\  t\in\R.$ 
Since \eqref{d1} and \eqref{d2} reduce to 
\begin{equation*}
\begin{cases}
\d X(t)=\ll\, Y(t)\d t\\
\d Y(t)=\{b(X_t,Y_t)-\ll (X(t)-\bar X(t))-2\ll \beta (Y(t)-\bar Y(t))\}\d t+\si(X_t,Y_t)\d \tilde W(t)
\end{cases}
\end{equation*}
and 
\begin{equation*}
\begin{cases}
\d \bar X(t)=\ll\, \bar Y(t)\d t\\
\d \bar Y(t)=b(\bar X_t, \bar Y_t)\d t+\si(\bar  X_t,\bar  Y_t)\d \tilde W(t),
\end{cases}
\end{equation*}
by It\^o's formula we obtain
\begin{equation}\label{d3}
\begin{split}
&\d V(Z(t))
=\Big\{\big\<(1+2\bb^2)(X(t)-\bar X(t))+\bb(Y(t)-\bar  Y(t)), \ll(Y(t)-\bar Y(t))\big\>\\
&\quad+\big\<\bb(X(t)-\bar X(t))  +Y(t)-\bar Y(t),-\ll(X(t)-\bar X(t))-2\ll\bb(Y(t)-\bar Y(t))\big\>\\
&\quad+\big\<\bb(X(t)-\bar X(t))  +Y(t)-\bar Y(t),b(X_t,Y_t)-b(\bar X_t,\bar Y_t)\big\>\\
&\quad+\ff{1}{2}\|\si(X_t,Y_t)-\si(\bar X_t,\bar Y_t)\|_{HS}^2\Big\}\d t  + \d M(t)\\
&=: I(t)\d t+ \d M(t),\ \ t\ge 0,
\end{split}
\end{equation}
where
\begin{equation*}
\d M(t):=\big\<Y(t)-\bar Y(t)+\bb(X(t)-\bar
X(t)),(\si(X_t,Y_t)-\si(\bar X_t,\bar Y_t))\d \tt W(t)\big\>,
\end{equation*}
and by ({\bf C1}) and ({\bf C2}),
\begin{equation*}
I(t) \le-\ll \bb|Z(t)|^2 +(L_1+L_2/2) \|Z_t\|_r^2,\ \ t\ge 0.
\end{equation*}
Whence,  it follows from \eqref{d3}   that
\begin{equation}\label{s3}
\d V(Z(t)) \le \big\{-\ll \bb |Z(t)|^2  +(L_1+L_2/2)
\|Z_t\|_r^2\big\}\d t+ \d M(t),\ \ t\ge 0.
\end{equation}
Letting  $\ll'=\ff{\ll\,\bb}{2\,c_\bb}$ such that
$2c_\bb\ll'-\ll\bb=0$,
 and combining this with \eqref{*SV},
we obtain \beg{align*} \d (\e^{2\ll' t} V(Z(t)))
 &= \e^{2\ll' t}\{2\ll'  V(Z(t))\d t+ \d  V(Z(t))\}\\
 &\le \e^{2\ll' t}\{(2c_\bb \ll' - \ll\bb) |Z(t)|^2 +(L_1+L_2/2)\|Z_t\|_r^2\}  \d t +\e^{2\ll' t} \d M(t)\\
 &=(L_1+L_2/2)\e^{2\ll' t} \|Z_t\|_r^2  \d t +\e^{2\ll' t} \d M(t),\ \ t\ge 0.
 \end{align*}
Setting $\kk=2(\ll'-r)$ and using \eqref{*SV} again,  we derive that
 \begin{equation}\label{s4}
 \begin{split}
 \e^{2rt}  |Z(t)|^2   &\le 4\e^{-\kk t} V(Z(0)) +4 (L_1+L_2/2) \int_0^t \e^{-\kk(t-s) }  \e^{2rs}\|Z_s\|_r^2 \d
 s\\
 &\quad+4\e^{-\kk t}\int_0^t \e^{2\ll's}\d M(s).
 \end{split}
 \end{equation}
For any $k>\|\xi\|_r+\|\eta\|_r+\|\bar\xi\|_r+\|\bar\eta\|_r$,
define the stopping time
\begin{equation*}
\tau_k=\inf\{t\ge0:\|X_t\|_r+\|Y_t\|_r+\|\bar X_t\|_r+\|\bar
Y_t\|_r\ge t\}.
\end{equation*}
By the H\"older inequality, one has
\begin{equation}\label{s5}
\Big(\int_0^{t\wedge\tau_k}\e^{-\kk(t\wedge\tau_k-
s)}\e^{2rs}\|Z_s\|_r^2\d s\Big)^{p_0}
\le\ff{(1-1/p_0)^{p_0-1}}{\kk^{p_0-1}}\int_0^{t\wedge\tau_k}\e^{2p_0rs}\|Z_s\|_r^{2p_0}\d
s.
\end{equation}
Moreover, employing \cite[Lemma 2.2]{E09}  leads to
\begin{equation}\label{s6}
\begin{split}
&\E_\Q\bigg(\sup_{0\le s\le t\wedge\tau_k }\Big(\e^{-\kk s} \int_0^s
\e^{2\ll'u}\d
M(u)\Big)^{p_0}\bigg)\\
&\le\ff{\Lambda_{p_0,\aa_0}L_0^{p_0/2}}{\kk^{p_0/2-1}}\int_0^t\E_\Q(\e^{2p_0r(s\wedge\tau_k))}\|Z_{s\wedge\tau_k}\|_r^{2p_0})\d
s,
\end{split}
\end{equation}
where  the explicit expression of $\Lambda_{p_0,\aa_0}$ was provided
in the last line of the argument of \cite[Lemma 2.2]{E09}. Thus,
taking \eqref{s4}, \eqref{s5}, and \eqref{s6} into account and
employing Fatou's lemma yields
\begin{equation}
\E(\e^{2p_0rt}\|Z_t\|_r^{2p_0})\le
c_{p_0,\vv}\|Z_0\|_r^{2p_0}+(1+\vv)\ff{\mu_{p_0}}{\kk^{p_0/2-1}}\int_0^t\E_\Q(\e^{2p_0rs}\|Z_s\|_r^{2p_0})\d
s,~~\vv>0,
\end{equation}
for some constant $c_{p_0,\vv}>0,$ where $\mu_{p_0}$ was introduced
in \eqref{s8}. Consequently, the desired assertion follows by taking
$\vv>0$ sufficiently small, applying   Gronwall's inequality and
utilizing \eqref{s7}.

\end{proof}

\appendix

\section{Appendix}\label{app}
To make the content self-contained, in this section, we  address
existence and uniqueness of solutions to \eqref{eq0} under the
locally weak monotonicity and the weak coercivity. Assume that
\begin{itemize} \item[({\bf D1})]  $b\in C(\C_r)$ and $\si\in C(\C_r)$ are   bounded on bounded subsets of $\C_r$, and, for each $k\ge1$, there is
an $L_k>0$ such that for all $\xi,\eta\in\C_r$ with
$\|\xi\|_r\vee\|\eta\|_r\le k$,
\begin{equation*}
2\<\xi(0)-\eta(0), b(\xi)-b(\eta)\>+\|\si(\xi)-\si(\eta)\|_{\rm
HS}^2\le L_k\|\xi-\eta\|_r^2.
\end{equation*}
\item[({\bf D2})]  There exists an  $L>0$ such that
$ 2\<\xi(0),b(\xi)\>^++\|\si(\xi)\|_{\rm HS}^2\le L(1+\|\xi\|_r^2),
~\xi\in\C_r. $
\end{itemize}

\begin{thm}\label{existence}
{\rm Let
  $({\bf D1})$ and $({\bf D2})$ hold. Then,
  $\eqref{eq0} $ has a unique solution $(X(t))_{t\ge0}$ such that  for
  some $C>0$,
 \begin{equation}\label{e9}
 \E\|X_t\|_r^2\le C \e^{C\,t}(1+ \|\xi\|_r^2),\ \ t\ge 0,
 ~~\xi\in\C_r.
 \end{equation}
}
\end{thm}

\begin{proof}
Below we follow the idea of \cite[Theorem 2.3]{VS}.  Set $N_0:=
\{n\in \mathbb N: n\ge \ff{r}{\log 2}\}$  and  $[s]:=\sup\{k\in
\mathbb Z: k\le s\}$, the integer par of $s>0$. For any $n\in N_0$,
consider an SDE
\begin{equation}\label{ee1}
\d X^n(t)=b(\hat X_t^n)\d t+\si(\hat X_t^n)\d
W(t),~~~t>0,~~~X_0^n=X_0=\xi,
\end{equation}
where,  $ \hat X_t^n(\theta):=X^n((t+\theta)\wedge t_n),
\theta\in(-\8,0]$ and $t_n:= [nt]/n. $
  Define the
 stopping time
\begin{equation}\label{ee2}
\tau^n_R=\inf\Big\{t\ge0:\ |X^n(t)|\ge R\Big\} =\inf\Big\{t\ge 0:\
\|X^n_t\|_r\ge R\Big\},~R> \|\xi\|_r,~n\in N_0.
\end{equation}
Thanks to $n\in N_0$, we have  $\e^{r/n}\le 2$ so that
\begin{equation}\label{ee3}
\|\hat X_t^n\|_r\le  \|X_t^n\|_r\lor |X^n(t_n)|\le \e^{r(t-t_n)}
\|X_t^n\|_r \le 2 \|X_t^n\|_r.
\end{equation}
 Since $b$ is bounded on  bounded subsets of $\C_r$, we
get
\begin{equation}\label{ee4}
|b(X_t^n)|\le C(R):=\sup_{\|\zeta\|_r\le R}|b(\zeta)|<\8,\ \ R\in
(\|\xi\|_r,\infty), \,t\in [0, \tau_R^n].
\end{equation}
Let $Z^{n,m}(t)=X^n(t)-X^m(t)$ and $ p^n_t=X^n_t-\hat X^n_t.$
 By the notion of $\tau^n_R$,
 \eqref{ee3} implies that \beq\label{GPP0} \|p^n_t\|_r\le 3R,\ \ t\le
\tau_R^n.\end{equation} By It\^o's formula and using ({\bf D1}),
\eqref{ee3} and \eqref{ee4}, there are $C,K>0$ such that
\begin{equation*}
\begin{split}
&\d(\e^{2rt}|Z^{n,m}(t)|^2)\le K \Big\{\sup_{0\le s\le t} (\e^{2rs}|Z^{n,m}(s)|^2)+
\e^{2rt}(\|p^n_t\|_r+\|  p^m_t\|_r)\Big\}\d t+\d M^{n,m}(t)
\end{split}
\end{equation*}
for any $t\in[0,\tau_R^n\wedge\tau_R^m]$, where
$
\d M^{n,m}(t):=2\,\e^{2rt}\big\<Z^{n,m}(t),(\si(\hat X_t^n)-\si(\hat
X_t^m))\d W(t)\big\>.
$
By the stochastic Grownwall inequality \cite[Lemma 5.4]{VS}, for any
$T>0,$   $p\in(0,1)$ and $q>\ff{1+p}{1-p}$, there exists a constant
$c_1>0$ such that
\begin{equation}\label{ee5}
\begin{split}
\E\Big(\sup_{0\le t\le
T\wedge\tau^n_R\wedge\tau^m_R}(\e^{2rt}|Z^{n,m}(t)|^2)\Big)^p&\le
c_1\bigg(\int_0^T\E(\|p_t^n\|_r^q{\bf1}_{\{t\le\tau_R^n\}})\d
t\bigg)^{p/q}\\
&\quad +c_1
\bigg(\int_0^T\E(\|p_t^m\|_r^q{\bf1}_{\{t\le\tau_R^m\}})\d
t\bigg)^{p/q}.
    \end{split}
\end{equation}
A straightforward calculation   leads to
\begin{equation}\label{PP1}
\begin{split}
&\|p_t^n\|_r\le \int_{t_n}^t|b(\hat X_s^n)|\d s+\sup_{t_n\le s\le t}
\Big|\int_{t_n}^s\si(\hat X_s^n)\d W(s)\Big|.
\end{split}
\end{equation}
From \eqref{ee4} and by the local  boundedness of $\si$ and   BDG's
inequality, for some $M_R>0$,
\begin{equation}\label{PP2}
\begin{split}
&\lim_{n\to\infty} \E\Big(\int_{t_n}^{t\land \tau_R^n} |b(\hat
X_s^n)|\d s\Big)^q+\E \bigg(\sup_{t_n\le s\le t\land \tau_R^n}
\Big|\int_{t_n}^s\si(\hat X_s^n)\d W(s)\Big|^q
\bigg)\\
&\le\lim_{n\to\infty} \Big(\ff{C(R)}{n^q}+\ff{M_R}{n^{q/2}}\Big)
=0,\ \ t\ge 0.
\end{split}
\end{equation}
 Combining this with  \eqref{PP1},    we make a conclusion  that
\begin{equation}\label{ee6}
\sup_{t\in[0,T]}\lim_{n\to\8}\E(\|p_t^n\|_r^q{\bf1}_{\{t\le\tau_R^n\}})=0,
\end{equation} which, together with   \eqref{ee5}  for $p=\ff 1 2$,
implies that
\begin{equation}\label{ee11}
\lim_{n,m\rightarrow\8}\E \Big\{\sup_{0\le t\le
T\wedge\tau^n_R\wedge\tau^m_R}\|X^n_t-X^m_t\|_r\Big\}=0.
\end{equation}
So, to ensure that $X_\cdot^n$ converges in probability to a
solution of \eqref{eq0},   it remains to prove
\begin{equation}\label{ee9}
\lim_{R\to\8}\limsup_{n\to\8}\P(\tau^n_R\le T)=0.
\end{equation}
  Indeed, \eqref{ee11}
and \eqref{ee9} yield that
\begin{equation*}
\lim_{n,m\rightarrow\8}\P\Big\{\sup_{0\le t\le
T}\|X^n_t-X^m_t\|_r\ge\vv\Big\}=0,  \ \ \vv>0,
\end{equation*}
and thus, due to the completeness of $(\C_r,\|\cdot\|_r)$,   there
exists a continuous adapted process $(X_t)_{t\in [0,T]}$ on $\C_r$
such that
\begin{equation*}\label{GPP1}
\sup_{0\le t\le T}\|X^n_t-X_t\|_r\rightarrow0~~~~\mbox{ in
probability as } n\rightarrow\8.
\end{equation*}
Subsequently, by carrying out a standard argument, we can show that
$(X(t))_{t\in [0,T]}$ is the unique functional solution to
\eqref{eq0} under assumptions ({\bf D1}) and ({\bf D2}). We now
proceed to verify   \eqref{ee9}. By It\^o's formula, besides $({\bf
D2})$,
   there is   a constant   $c_2 >0$ such that
\begin{equation}\label{eq3}
\begin{split}
 \d(\e^{2rt}|X^n(t)|^2)  &\le c_2\e^{2rt}\Big\{ 1+ |X^n(t)|^2+4\|  X^n_t\|_r^2 +\| p^n(t)\|_r\cdot|b(\hat X^n_t)|\Big\}\d t
  +\d M^n(t),
\end{split}
\end{equation}
where $\d M^n(t):=2\, \e^{2rt}\<X^n(t),\si(\hat X^n_t)\d W(t)\>.$
 By combining \eqref{ee4}  and using ({\bf D2}),
BDG's inequality and Gronwall's inequality, there exists a constant
$c_3>0$ such that
 \beq\label{GPP2}\beg{split}\GG^{n,R}(t):&=\E\bigg(\sup_{0\le s\le
 t\wedge\tau^n_R}(\e^{2rs}|X^n(s)|^2)\bigg)\\
 &\le c_3\,\e^{ c_3t}\bigg\{\|\xi\|_r^2+t+
\int_0^t\e^{2rs}\E\Big(|p^n(s)|{\bf 1}_{\{s\le \tau^n_R\}}\Big)\d
s\bigg\},\ \ t\ge 0
\end{split}\end{equation}
holds for some constant $c_3>0$. Next, \eqref{ee6}, \eqref{GPP2} and   Chebyshev's inequality  gives
\beg{align*}&\lim_{R\to\infty}\lim_{n\to \infty} \P(\tau_R^n\le
T)=\lim_{R\to\infty}\lim_{n\to \infty} \P\Big(\tau_R^n\le
T,\sup_{0\le t\le\tau_R^n\land T
}|X^n(t)|\ge\ff{R}{4}\Big) \\
& \le\lim_{R\to\infty}\lim_{n\to \infty}  \P\Big(\sup_{0\le
t\le\tau_R^n\wedge T }|X^n(t)|\ge\ff{R}{4}\Big)
 \le \lim_{R\to\infty}\lim_{n\to \infty} \ff{16\Gamma^{n,R}(T)}{R^2} =0,\end{align*}
where we used the fact that \begin{equation*} \Big\{\tau_R^n\le
T,\sup_{0\le t\le\tau_R^n\land T
}|X^n(t)|<\ff{R}{4}\Big\}=\emptyset,
\end{equation*}
by the definition of $\tau_R^n.$ So,   \eqref{ee9} holds.

In the end, by making use of \eqref{ee6}
  and \eqref{GPP2} and employing Fatou's lemma for
$n\to\infty$, we obtain
$$\E\bigg(\sup_{0\le s\le t\wedge\tau_R}(\e^{2rs}\|X_s\|^2_r)\bigg)
 \le c_4\big(1+\|\xi\|_r^2\big)\e^{c_4t},$$
where $\tau^R$ is defined as in \eqref{ee2} for $X$ replacing $X^n$,
which goes to $\infty$ as $R\to\infty$. Therefore, by approaching
$R\uparrow\infty$, we achieve \eqref{e9}.
\end{proof}


\end{document}